\documentclass[11pt]{article}
\usepackage{amsmath, amssymb, amsthm, amscd}
\numberwithin{equation}{section}

\def\p{\partial}

\def\l{\lambda}

\def\cC{{\cal C}}

\def\cH{{\cal H}}

\def\R{{\mathbb R}}

\def\C{{\mathbb C}}

\def\cC{{\mathcal C}}

\def\cH{{\mathcal H}}

\def\R{{\mathbb R}}
\def\C{{\mathbb C}}

\newtheorem{prop}{Proposition}[section]
\newtheorem{theo}[prop]{Theorem}
\newtheorem{lemma}[prop]{Lemma}
\newtheorem{cor}[prop]{Corollary}
\newtheorem{rmk}[prop]{Remark}

\newtheorem{conj}[prop]{Conjecture}

\let\lra=\longrightarrow

\def\mapright\#1{\,\smash{\mathop{\lra}\limits^{\#1}}\,}

\begin{document}
\title{A class of fully nonlinear equations}
\author{Xiuxiong Chen, Weiyong He}

\maketitle

\begin{abstract}In this paper we consider a class of fully nonlinear equations which covers the equation introduced by S. Donaldson a decade ago and the equation introduced by Gursky-Streets recently. We solve the equation with uniform weak $C^2$ estimates, which hold for degenerate case. 
\end{abstract}

\section{Introduction}
We recall a class of differential operators introduced by S. Donaldson  \cite{Donaldson2007} and Gursky-Streets \cite{GS}.
Consider a function $u: \R\times \R^n\rightarrow \R$ with the coordinate $(t, x)$. We use the operator $D=(\p_t, \nabla)$ to denote the first order derivatives. Consider the matrix 
\[
r=\begin{pmatrix}
u_{tt} & \nabla u_t\\
(\nabla u_t)^t &R
\end{pmatrix}
\]
where $R=\nabla^2 u+\text{lower order terms}$. 
Given a symmetric matrix $P$, we use $\sigma_i(P)$ to denote the $i$-th elementary symmetric function on its eigenvalues $\_1, \cdots, \l_n$. The $\Gamma_k^+$ cone is 
\[
\Gamma_k^+=\{P: \sigma_i(P)>0, 1\leq i\leq k\}. 
\]
Assume $u_{tt}>0$ and $R\in \Gamma_k^+$, consider the operator
\begin{equation}
F_k(r)=u_{tt}\sigma_k(R)-(T_{k-1}(R), \nabla u_t\otimes \nabla u_t),
\end{equation}
where $T_{k-1}$ is the $(k-1)$-th Newton transformation which takes the form of
\[
T_{k-1}(R)_{ij}=\sigma_k(R)\frac{\p }{\p{R^{ij}}}\log \sigma_k(R). 
\]
This operator appears naturally in two different settings of geodesic equations of certain infinite dimensional Riemannian geometry. 

When $k=1$, the operator was introduced by S. Donaldson \cite{Donaldson2007}
\[
F_1(r)=u_{tt}(\Delta u+1)-|\nabla u_t|^2, 
\]
when he considered a Weil-Peterson type
metric on the space of volume forms (normalized) on a Riemannian
manifold $(X, g)$ with fixed total volume. This infinite
dimensional space can be parameterized by all smooth functions
such that
\[
\{\phi\in  C^\infty (X): 1 + \triangle_g \phi > 0\}.
\]
The  metric is defined by
\[
\|\delta \phi\|_\phi^2 =  \int_X\; (\delta \phi)^2 (1 +
\triangle_g \phi) dg.
\]
Then the geodesic equation is
\begin{equation}\label{E-1-1}
u_{tt}(1+\triangle u) -|\nabla u_t|^2_g = 0.
\end{equation}

For all $k\geq 1$, Gursky-Streets \cite{GS} introduced a family of operators $F_k$. 
Consider a conformal class $g_u=e^{-2u}g$ on  a Riemannian manifold $(M, g)$.  Recall the Schouten tensor
\[
A:=\frac{1}{n-2}\left(Ric-\frac{1}{2(n-1)} Rg\right),
\]
which plays an important role in conformal geometry. Under the conformal change, the Schouten tensor is given by
\[
A_u=A(g_u)=A+\nabla^2 u+\nabla u\otimes \nabla u-\frac{1}{2}|\nabla u|^2 g.
\]
When $A_u\in \Gamma_k^+$, Gursky-Streets introduced a family of fully nonlinear elliptic equations of the form
\[
u_{tt}\sigma_k(A_u)-(T_{k-1}(A_u), \nabla u_t\otimes \nabla u_t)=0.
\]
When $n=4$, $k=2$, 
this is the geodesic equation  of the following metric
\[
\langle \psi, \phi\rangle_{u}=\int_M \phi \psi \sigma_2(g_u^{-1}A_u) dV_u, 
\]
defined on the space $\cC^{+}=\{u: A_{g_u}\in \Gamma_2^{+}, g_u=e^{-2u}g\}$. Gursky and Streets  introduce these structures to solve the uniqueness of $\sigma_2$ Yamabe problem on a four Riemannian manifold. 
We refer the readers to \cite{GS, He17} for more details. 
When $k=1$, the Gursky-Streets equation reads
\[
u_{tt}(\Delta u-(n/2-1)|\nabla u|^2+A(x))-|\nabla u_t|^2=0.
\] 
The Donaldson equation and the Gursky-Streets equation are closely related in this case.
In this paper we discuss  a class of equations of the following form,  
\begin{equation}\label{GS1}
u_{tt}\left(\Delta u- b|\nabla u|^2+a(x)\right)-|\nabla u_t|^2=f,
\end{equation}
with boundary condition
\[u(\cdot, 0)=u_0, u(\cdot, 1)=u_1,\]
where $a(x): M\rightarrow \R$ is a positive smooth function and $b$ is a nonnegative constant. 
We define the function space
\[
\cH=\{\phi\in C^\infty(M), \Delta \phi-b |\nabla \phi|^2+a(x)>0\}
\]
and $u_0, u_1\in \cH$. Note that the sign $-b|\nabla u|^2$ makes the space $\cH$ convex, meaning that if $u_0, u_1\in \cH$, then $(1-t)u_0+tu_1\in \cH$ for any $t\in [0, 1]$.  \\

A main result of the paper is the following,
\begin{theo}\label{T-1-1}
Let $(M, g)$ be a compact Riemannian manifold and $f\in
C^{k}(M\times [0, 1])$ with $k\geq 2$ is a positive function. The
Dirichlet problem \eqref{GS1} has a unique  solution $u(x,
t)\in C^{k+1, \beta}(M\times [0, 1])$ for any  $\beta\in [0, 1)$.
The uniform $C^1$ estimates and estimates of  $u_{tt}, |u_{tk}|, \Delta u$ do not depend on $\inf f$, but on $(M, g)$, boundary datum $u_0, u_1$ and 
\[
\max\left\{\sup f, \sup |Df^{1/2}|, \sup {|f_{tt}|}, \sup {|\Delta f|}\right\}
\]
 for any $t\in [0, 1]$.
\end{theo}

\begin{rmk}This generalizes the results in \cite{Chen-He}, where the authors solved the Donaldson equation with righthand side $\epsilon$. Here we consider a class of equations which also covers the Gursky-Streets equation when $k=1$.  Our computations are much more streamlined and simplified. 
\end{rmk}

As a direct corollary, we solve the homogeneous equation with the weak $C^2$ bound.

\begin{cor}\label{C1}Let $(M, g)$ be a compact Riemannian manifold. Then there exists a solution to the Dirichlet problem of the  homogeneous equation
\[
u_{tt}(\Delta u-b|\nabla u|^2+a(x))-|\nabla u_t|^2=0
\]
such that $u(0, \cdot)=u_0$ and $u(1, \cdot)=u_1$ with the uniform bound,
\[
|u|_{C^1}+|u_{tt}|+|\Delta u|+|\nabla u_t|\leq C. 
\]
\end{cor}

{\bf Acknowledgement}: The first author is supported in part by an NSF fund. The second author is supported partly by an NSF fund, award no. 1611797.
\section{Solve the equation}
For simplicity, we write
\[
B_u=\Delta u-b |\nabla u|^2+a(x).
\]
Its linearized operator is given by
\[
L_{B_u} (h)=\Delta h-2b (\nabla u, \nabla h)
\]
We write the equation
\begin{equation}\label{E-1-4}
Q{(u_{tt}, B_u, \nabla u_t)}:=u_{tt}B_u-|\nabla u_t|^2=f,
\end{equation}
where $f\in C^{\infty}(M\times [0, 1])$ is a positive function and
$u_0, u_1\in \cH$.  When there is no confusion, we also write 
\[
Q(u)=Q{(u_{tt}, B_u, \nabla u_t)}
\]
We compute the linearized operator, which is given by
\[
\begin{split}
dQ(h)=&u_{tt}[\triangle h-2b(\nabla u, \nabla h)]+B_u h_{tt}-2\langle
\nabla h_t, \nabla u_t\rangle\\
=& u_{tt}L_{B_u}(h)+B_u h_{tt}-2\langle
\nabla h_t, \nabla u_t\rangle.
\end{split}
\]
We will use the following notations. At any point $p\in M\times
[0, 1]$, take local coordinates $( x_1, \cdots, x_n, t)$. We can
always diagonalize the metric tensor $g$ as
$g_{ij}(p)=\delta_{ij}, \p_kg_{ij}(p)=0$. We will use, for any
smooth function $f$ on $X\times [0, 1]$, the following notations
\[
\triangle f_i=\triangle (f_i), ~~\triangle f_{ij}=\triangle
(f_{ij}), ~~\triangle f_{,i}=(\triangle f)_{, i}~~\mbox{and}~~
\triangle f,_{ij}=(\triangle f)_{ij}.
\]
For any function $f,$ $f_i, f_{ij}$ etc are covariant derivatives.
By Weitzenbock formula, we have
\begin{equation}\label{E-2-2}
\triangle f_i=\triangle f,_{i}+R_{ij}f_j,
\end{equation}
where $R_{ij}$ is the Ricci tensor of the metric $g$.\\

The following concavity is important for solving the equation. 

\begin{lemma}[Donaldson \cite{Donaldson2007}] \label{L-2-1}1. If $A>0$, then
$Q(A)>0$ and if $A\geq 0$, $Q(A)\geq 0$.\\

2. If $A, B$ are two matrices with $Q(A)=Q(B)>0,$ and if the
entries $A_{00}, B_{00}$ are positive then for any $s\in [0, 1]$,
\[
Q(sA+(1-s)B)\geq Q(A), Q(A-B)\leq 0.
\]
Moreover, strict inequality holds if the corresponding arguments
are not the same.
\end{lemma}
We have its equivalent form.
\begin{lemma}[\cite{Chen-He}] \label{L-2-2} Consider the function
\[
f(x, y, z_1, \cdots, z_n)=\log{\left(xy-\sum z_i^2\right)}.
\]
Then $f$ is concave when $x>0, y>0, xy-\sum z_i^2>0$.
\end{lemma}

First we assume $u$ solves the Dirichlet problem \eqref{GS1} and derive the \emph{a priori} estimates. With these estimates, it is standard to use the method of continuity to solve the equation. 
\subsection{$C^{0}$ estimates and uniqueness}
Denote $U_c=ct(1-t)+(1-t)u_0+tu_1$ for any number $c$.
\begin{lemma}\label{L-2-3}
For some $c>0$ big enough,
\[
U_{-c}\leq u\leq (1-t)u_0+t u_1.
\]
Moreover, the solution $u$ is unique. 
\end{lemma}
\begin{proof}First we have
\[u_{tt}>0.
\]
It follows that
\[
\frac{u(\cdot, t)-u(\cdot, 0)}{t-0}< \frac{u(\cdot,
1)-u(\cdot, t)}{1-t}.
\]
Namely
\[
u(t)<(1-t)u_0+tu_1.
\]
Note that $u=U_{-c}$ on the boundary. If $u<U_{-c}$ for
some point, then $v=u-U_{-c}$ obtains its minimum in the
interior, say at $p$. Then $, \nabla v=0, D^{2}v\geq 0$ at $p$. By the concavity of $\log Q$, we have
\begin{equation}\label{e0}
Q^{-1}dQ (v)\leq \log Q(u)-\log Q(U_{-c}),
\end{equation}
where $Q^{-1}dQ$ takes value at $u$. Clearly $Q(U_{-c})=2c B_{U_{-c}}-|\nabla u_0-\nabla u_1|^2$. 
Note that $B_{U_{-c}}\geq (1-t)B_{u_0}+tB_{u_1}$ is strictly positive. 
If we choose $c$ sufficiently large, the righthand side of \eqref{e0} is negative. However at $p$, $\nabla v=0, D^2 v\geq 0$, we claim
$dQ(v)\geq 0. $ Contradiction.
 To see the claim, we choose a vector $(x_0, Y)=(x_0, y_1, \cdots y_n)$, then by $D^2v (p)\geq 0$ we have, 
\[
v_{tt} x_0^2-2x_0 (\nabla v_t, Y)+Y \nabla^2 v Y^t\geq 0
\]
Choose $x_0=B_u, Y=\nabla u_t$ and note $Y \nabla^2 v Y^t\leq \Delta v|\nabla u_t|$. It follows 
\[
2(\nabla u_t, \nabla v_t)\leq v_{tt}B_u+B_{u}^{-1}\Delta v |\nabla u_t|^2 
\]
We compute $dQ(v)=v_{tt}B_u+u_{tt}(\Delta v+a(x))-2(\nabla u_t, \nabla v_t)\geq a(x)u_{tt}>0.$  
The same argument gives the uniqueness.
\end{proof}

\subsection{$C^1$ estimates}  
\begin{prop}
We have the following,
\[-c+u_1-u_0\leq  u_{t}(0, \cdot)\leq u_1-u_0\leq  u_t(1, \cdot)\leq u_1-u_0+c.\]
\end{prop}
\begin{proof} By Lemma \ref{L-2-3},
\[
-ct(1-t)+(1-t)u_0+tu_1\leq  u\leq (1-t)u_0+tu_1.
\]
Since $u_{tt}>0$,  $u_t$ obtains its
maximum on the boundary. It is then easy to verify that the estimate holds. 
\end{proof}

\begin{rmk}
Since $ u+At+B$ still solves the equation for any constants $A, B$. The boundary data changes as , $u_0\rightarrow u_0+B$, $u_1\rightarrow u_1+A+B$ and $ u_t\rightarrow  u_t+A$.  (Note that $\nabla  u$ remains the same.)
Since we have uniform bound on $| u|_{C^0}$ and $|u_t|$, we can choose $A, B$ accordingly such that $1\leq | u_t|\leq C$, and $1\leq -u\leq C$. We assume this normalization in the following. 
\end{rmk}

We need some preparations. We have the following straightforward computations.

\begin{prop}We have
\[
dQ(t)=0, dQ(t^2)=2B_u. 
\]
\end{prop}

\begin{prop}We have
\[dQ(u)=2f-(a+b|\nabla u|^2)u_{tt}\]
\end{prop}

\begin{proof}
We compute
\[dQ(u)=u_{tt}(\Delta u-2b|\nabla u|^2)+B_u u_{tt}-2|\nabla u_t|^2\]
Using the equation this completes the proof.
\end{proof}

\begin{prop}
Given $\phi, \psi$, we have
\begin{equation}\label{q0}
dQ(\phi\psi)=\psi dQ(\phi)+\phi dQ(\psi)+2q_u(D\phi, D\psi),
\end{equation}
where the quadratic form is given by
\[
q_u(D\phi, D\psi)=u_{tt}(\nabla \phi, \nabla \psi)+B_u(\phi_t, \psi_t)-(\nabla u_t, \phi_t\nabla \psi+\psi_t\nabla \phi)
\]
Note that $q_u(D\phi, D\phi)\geq 0$. 
\end{prop}

\begin{prop}We compute
\begin{equation}\label{q1}
dQ(|\nabla u|^2)=2u_{tt}(R_{ij}u_iu_j-a_iu_i)+2f_iu_i+2q_u(\nabla u_i, \nabla u_i)
\end{equation}
\end{prop}
\begin{proof}We compute,
\[dQ(u_i)= u_{tt}(\Delta u_i-2b (\nabla u, \nabla u_i))+B_uu_{tti}-2u_{tk}u_{tki}\]
Taking derivative of the equation, we get
\[
u_{tt}((\Delta u)_i-2b (\nabla u, \nabla u_i)+a_i)+B_uu_{tti}-2u_{tk}u_{tki}=f_i. 
\]
It follows that
\begin{equation}\label{g0}
dQ(u_i)=u_{tt}(R_{ij}u_j-a_i)+f_i.
\end{equation}
Applying \eqref{q0} to $\phi=u_i$, we get  \eqref{q1}. 
\end{proof}

\begin{lemma}There exists a uniform constant $C_2=C_2(g, |u_0|_{C^1}, |u_1|_{C^1}, \sup f, |\nabla f^{1/2}|)$ such that
\[
|\nabla u|\leq C_2. 
\]
\end{lemma}
\begin{proof}
To bound $\nabla
u$, take
\[
h=\frac{1}{2}\left(|\nabla u|^2+\l u^2\right),
\]
where $\l$ is a constant determined later. We want to show that $h$
is bounded. Namely, there exists a constant $C_1$ depending only on
$\sup f, |\nabla f^{1/2}|$ and the boundary data such that
\[ \max h\leq C_2.
\]
Since $h$ is uniformly bounded on the boundary,  we assume  $h$
takes its maximum at $(p, t_0)\in M\times (0, 1)$. 
We compute
\[dQ(u^2)=-2u(a+b|\nabla u|^2)+2fu+2q_u(Du, Du)\]
It follows that, using \eqref{q1},
\begin{equation}\label{q2}
\begin{split}dQ(h)=&u_{tt}(R_{ij}u_iu_j-a_iu_i)+f_iu_i+q_u(\nabla u_i, \nabla u_i)\\&-\l u (a+b|\nabla u|^2)+\l fu+\l q_u(Du, Du)\\
\geq &-C_0u_{tt}(|\nabla u|^2+|\nabla u|)-|\nabla u| |\nabla f|\\
&-\l u (a+b|\nabla u|^2)+\l fu+\l q_u(Du, Du),
\end{split}
\end{equation}
where $C_0$ depends on $\max |Ric|$ and $|\nabla a|$. 
At the point $p$, since $Dh =0$, we have
\[
h_t=u_k u_{tk}+\l u u_t=0
\]
We compute
\[
\begin{split}
q_u(Du, Du)=&u_{tt}|\nabla u|^2+B_u u_t^2-2u_{tk} u_t u_k\\
=&u_{tt}|\nabla u|^2+B_u u_t^2+2\l u u_t^2
\end{split}
\]
If $b>0$, we compute
\[
dQ(h)> u_{tt}(\l |\nabla u|^2-C_0|\nabla u|^2-C_0|\nabla u|)+\l b |\nabla u|^2-|\nabla u| |\nabla f|+\l fu,
\]
 At the point $p$ ($h$ achieves its maximum), $dQ(h)\leq 0$. This follows that
 \[
 u_{tt}(\l |\nabla u|^2-C_0|\nabla u|^2-C_0|\nabla u|)+\l b |\nabla u|^2-|\nabla u| |\nabla f|+\l fu\leq 0
 \]
 Hence  this gives the bound $|\nabla u|(p)\leq C_2$ if $\l$ is sufficiently large. 
 If $b=0$, we compute
 \[
 dQ(h)> -C_0u_{tt}(|\nabla u|^2+|\nabla u|)-|\nabla u| |\nabla f|+\l u_{tt}|\nabla u|^2+\l B_u |u_t|^2+\l f u. 
 \]
Note that
\[
u_{tt}|\nabla u|^2+B_u u_t^2\geq 2\sqrt{u_{tt}B_u} |\nabla u||u_t|\geq  2\sqrt{f} |\nabla u|
\]
We compute that
\[
dQ(h)>u_{tt}(\l |\nabla u|^2/2-C_0|\nabla u|^2-C_0|\nabla u|)+\l \sqrt{f}|\nabla u|-|\nabla f||\nabla u|+\l f u. 
\]
It follows that, at $p$,
\[
|\nabla u|(p)\leq C_2,
\]
where $C_2$ depends on $|\nabla f^{1/2}|$ in addition. This completes the proof.
\end{proof}

\subsection{$C^{2}$ estimates}
First we derive the boundary estimates. Due to the flatness of the boundary (in $t$ direction), the estimates of ``normal-normal" direction $ u_{tt}$ can be obtained from the equation that
\[
 u_{tt}\leq B_u^{-1} (|\nabla  u_t|^2+f),
\]
once the boundary estimates hold for $|\nabla  u_t|$. To bound the mixed term $|\nabla u_t|$ in the boundary estimates, we construct barrier functions  using similar ideas in  \cite{GuS, Guan}. The argument is purely local.
\begin{lemma} There exists a uniform constant $C_2$, such that at $t=0$ and $t=1$,
\[
 u_{tt}, |\nabla  u_t|\leq C_2
\]
where $C_2=C_2(g, |u_0|_{C^2}, |u_1|_{C^2}, |\nabla f^{1/2}|, \sup f)$
\end{lemma} 

\begin{proof}

We only argue for $t=0$. 
First we compute
\[
\begin{split}
dQ(u-u_0)&=-u_{tt}(\Delta u_0-b|\nabla u_0|^2+a)-bu_{tt}(|\nabla u|^2+|\nabla u_0|^2-2(\nabla u, \nabla u_0))\\
&\leq -u_{tt} B_{u_0}
\end{split}
\]
For a fix point $p\in M$, take a geodesic ball $B_r(p)\subset M$ around $p$ such that $r$ is less than injectivity radius. Consider the region \[U=\{(x, t)\in B_r(p)\times [0, 1]: d^2(x, p)+t^2\leq r^2\}\]
Take $A$ sufficiently large, and  denote
\[
h=A(u-u_0-ct)-A(t^2+d^2(x))+(\nabla u-\nabla u_0)_i,
\]
where $i=1, 2, \cdots, n$ and $d(x)=d(p, x)$ is the distance function. Note that $h$ is local function define on $\bar U$. 
We choose $c$ large enough such that $u-u_0-ct\leq 0$ and $B$ large enough such that 
$h\leq 0$ on $\p U$. We compute, using \eqref{g0},
\[
dQ((\nabla u-\nabla u_0)_i)\leq C_0 u_{tt}+|\nabla f|.
\]
Note that for $x\in B_r(p)$ for $r$ sufficiently small, 
\[dQ(d^2)=u_{tt}(\Delta d^2-2b(\nabla u, \nabla d^2))\geq 2u_{tt}(n-2bd |\nabla u|)>0.\] 
It then follows that
\[
dQ(h)\leq -A u_{tt}B_{u_0}-2A B_u+C_0 u_{tt}+|\nabla f|
\]
Choose $A B_{u_0}-C_0\geq 1$ and $A\sqrt{f}\geq |\nabla f|$, we get that
\[
dQ(h)\leq -u_{tt}-2AB_u+|\nabla f|\leq -2\sqrt{A f}+|\nabla f|\leq 0.
\]
By the maximum principle, it follows that $h\leq 0$ in $U$. Since $h(p, 0)=0$, it follows that
$\p_t h(p, 0)\leq 0$. 
Since $i$ and $p$ are arbitrary, 
this implies that $|\nabla u_t|(p, 0)\leq C_2$ at $t=0$, where $C_2$ depends on $|\nabla f^{1/2}|$ in particular. 
\end{proof}

Now we derive the interior $C^2$ estimates. 
We need some preparations to simply the computations. 
We write $r=(r_i)$ and 
\[
Q(r)=r_0 r_1-\sum_{i\geq 2} r_i^2,
\]
where $r=(u_{tt}, B_u, \nabla_i u_t)$. 
Then the equation $Q(r)=f$ can be written as 
$G(r)=\log f. $
Denote, for $0\leq i\leq n+1$, 
\[
Q^i=\frac{\p Q}{\p r_i}, Q^{i, j}=\frac{\p^2 Q}{\p r_i\p r_j}
\]
With this notation, we also record the linearization of $Q(r)$. 
We have
\begin{equation}\label{second0}
dQ(\psi)= u_{tt} (\Delta \psi-2b(\nabla u, \nabla \psi))+B_u \psi_{tt}-2 u_{tk}\psi_{tk} \end{equation}
If we write $(R_i)=(\psi_{tt}, L_{B_u}\psi, \nabla \psi_t)$, then
\[
dQ(\psi)=\sum_i Q^i R_i. 
\]

First we have the following interior estimates.
\begin{lemma}
There is a
uniform positive constants $C_2$ such that
\[
 u_{tt}\leq C_2,
\]
where $C_2=C_2( g, |u_0|_{C^2}, |u_1|_{C^2}, \sup f^{-1}|f_t|^2, \sup -f_{tt}, \sup f)$. 
\end{lemma}
\begin{proof}
We can compute by $G=\log Q=\log f$
\begin{equation}\label{1st1}
G^{i}\p_t r_{i}=Q^{-1}dQ( u_t)=f^{-1}f_t. 
\end{equation}
Taking derivative again, we have
\[
G^{i, j}\p_t r_{i}\p_t r_{j}+G^{i}\p_t^2 r_{i}=f^{-1}f_{tt}-f^{-2}f_t^2. 
\]
By concavity of $G$, we have
\[
G^{i}\p_t^2 r_{i}\geq f^{-1}f_{tt}-f^{-2}f_t^2. 
\]
Note that 
\[
\p_t^2 B_u=L_{B_u} (u_{tt})-2b|\nabla u_t|^2.
\]
It follows that we have
\[
G^i\p_t^2 r_i=Q^{-1}\left(dQ(u_{tt})-2b u_{tt}|\nabla u_t|^2\right)
\]
Hence we have
\[
dQ(u_{tt})\geq f_{tt}-f^{-1} f_t^2
\]
We compute
\[
dQ(u_{tt}-u)\geq (a+b|\nabla u|^2)u_{tt}+f_{tt}-f^{-1} f_t^2-2f
\]
If $u_{tt}-u$ takes the maximum at the boundary, then by the boundary estimate this is done. 
If the maximum appears interior, at the maximum point of $u_{tt}-u$, we have
\[
u_{tt}\leq C_3,
\]
where $C_3=C_3(\sup f, \sup -f_{tt}, \sup f^{-1} f_t^2)$.
This completes the proof. 
\end{proof}

Next we want to bound $\Delta  u$. We use the similar computation relying on the concavity of $G=\log Q$. 
\begin{lemma}There exists a uniform constant $C_4$ such that
\[
\Delta  u\leq C_4,
\]
where $C_4=C_4(g, |u_0|_{C^2}, |u_1|_{C^2}, \sup f, \sup -\Delta f, \sup D f^{1/2})$
\end{lemma}
\begin{proof}We only need to control the interior maximum. We compute
\begin{equation}\label{g1}
\begin{split}
&\nabla G=f^{-1}\nabla f=G^{i}\nabla r_{i},\\
& G^{i, j}\nabla r_{i} \nabla r_{j}+G^{i}\Delta r_{i}=f^{-1}\Delta f-f^{-2}|\nabla f|^2. 
\end{split}
\end{equation}
By the concavity of $G$, we have
\begin{equation}\label{s1}
Q^i\Delta r_i \geq \Delta f-f^{-1}|\nabla f|^2.
\end{equation}

There exists a difference between $Q^{i}\Delta r_{i}$ and $dQ(\Delta  u)$ coming from communication of covariant derivatives and the nonlinear term $-b|\nabla u|^2$. 
We compute
\[
(\Delta r_{i})=(\Delta u_{tt}, \Delta B_u, \Delta u_{tk})
\]
The Bochner-Weitzenbock identity gives
\[
\Delta |\nabla u|^2=2|\nabla^2 u|^2+2(\nabla \Delta u, \nabla u)+2Ric(\nabla u, \nabla u)
\]
Hence we have
\[
\begin{split}
\Delta B_u=&\Delta (\Delta u-b|\nabla u|^2+a)\\
=&L_{B_u} (\Delta u)-2b|\nabla^2 u|^2-2b Ric(\nabla u, \nabla u)+\Delta a
\end{split}
\]
We also have
\[
\Delta u_{tk}=Ric_{kj} u_{tj}+(\Delta u)_{tk}
\]
It follows that
\begin{equation}\label{s2}
\begin{split}
Q^i\Delta r_i=&dQ(\Delta u)-2b u_{tt} \left(|\nabla^2 u|^2+Ric(\nabla u, \nabla u)\right)\\
&-u_{tt}\Delta a-2Ric(\nabla u_t, \nabla u_t)
\end{split}
\end{equation}
Combining \eqref{s1} and \eqref{s2}, we have
\begin{equation}
\label{s3}
dQ(\Delta u)\geq 2b u_{tt} |\nabla^2 u|^2+2Ric (\nabla u_t, \nabla u_t)-C_2+\Delta f-f^{-1}|\nabla f|^2
\end{equation}
Since $b\geq 0$, the nonlinear term $-b|\nabla u|^2$ results in a good term $2b u_{tt}|\nabla^2 u|^2$. 
Now we denote $v=\Delta  u+\l t^2$. Then we have
\[
dQ(\Delta  u+\l t^2)\geq 2\l B_u-C_1 |\nabla  u_t|^2-C_2+\Delta f-f^{-1}|\nabla f|^2.
\]
Since $|\nabla u_t|^2\leq  u_{tt}B_u\leq CB_u$, we can choose $\l$ sufficiently large such that
\[
dQ(\Delta  u+\l t^2)\geq B_u-C_2+\Delta f-f^{-1}|\nabla f|^2.
\]
This is sufficiently to bound $\Delta  u$ from above. 
\end{proof}

To get higher regularity, we assume that $f$ is strictly positive. 
The H\"older estimate of $D^2 u$ follows from Evans-Krylov theory using  the concavity
of $\log Q$.  Once we get the H\"older estimates of $D^2 u$, the standard
boot-strapping argument gives all higher order derivatives of
$ u$. 

\subsection{Solve the equation}
To solve \eqref{GS1} for a general positive $f$, we consider the
following continuity family for $s\in [0, 1]$
\begin{equation}\label{E-3-1}
Q(u)=(1-s)Q(U_{-c})+sf,
\end{equation}
with the boundary condition
\[
 u(\cdot, 0, s)=u_0,  u(\cdot, 1, s)=u_1,
\]
 When $c$ is big
enough, $Q(U_{-c})$ is positive and bounded away from $0$.
We shall now prove that if $f\in C^{k}(X\times [0, 1])$ with
$k\geq 2$ then we can find of solution of \eqref{GS1} such that
$ u\in C^{k+1, \beta}(X\times [0, 1])$ for any $0\leq \beta<1$.
Consider the set
\[S=\left\{s\in[0, 1]:~\mbox{the equation (\ref{E-3-1}) has a solution in}~~C^{k-1, \beta}(X\times [0, 1])\right\}
\]
Obviously $0\in S$. Hence we need only show that $S$ is both open
and close. It is clear that $Q: C^{k+1, \beta}\rightarrow C^{k-1,
\beta}$ is open if \[ B_u>0~~ \mbox{and}~~
Q(u)>0.
\]
In this case $dQ$ is an invertible elliptic operator and openness
follows. The closeness of $S$ follows from the a prior estimates
derived in Section 2. Hence Theorem \ref{T-1-1} holds.

Since our estimates on $|u|_{C^1}, u_{tt}, \Delta u, |\nabla u_t|$ does not depend on $\inf f$, 
we can solve the equation
\[
Q(u)=s f
\]
for $s\in (0, 1]$ and $f>0$. Taking $s\rightarrow 0$,
this gives a strong solution of the homogeneous equation
\[
Q(u)=u_{tt}B_u-|\nabla u_t|^2=0,
\]
which has the uniform bound on $|u|_{C^1}, u_{tt}, \Delta u, |\nabla u_t|$. This proves Corollary \ref{C1}

\begin{rmk}
For the general righthand side $f\geq 0$ (possible degenerate) such that $|Df^{1/2}|$ is uniformly bounded, we can use an approximation argument to get a strong solution, by considering for example the equation
\[
u_{tt}(\Delta u-b|\nabla u|^2+a(x))=f+s
\]
for $s\in (0, 1]$.  Letting $s\rightarrow 0$ we get a strong solution. The only technical point is that uniqueness of homogeneous/degenerate equation does not follow directly from the comparison, which requires $f>0$. On the other hand, we believe that the uniqueness should still hold. 
\end{rmk}

\begin{rmk}
It would be interesting to see whether $|\nabla^2 u|$ is uniformly bounded, independent of $\inf f$. Such a result was proved for complex Monge-Ampere equation recently by \cite{CTW}. When $n=1$, the Donaldson equation is one special case of their results and it should work also for \eqref{GS1}. On the other hand, it would be interesting to see whether such an estimate holds for $n\geq 3$. 
\end{rmk}

\section{Discussions}
When $k=1$, the nonlinear term $-b|\nabla u|^2$ in $B_u=\Delta u-b|\nabla u|^2+a$ has the ``right" sign. Hence we can treat the Donaldson equation and the Gursky-Streets equation together. 
In \cite{Chen-He} only the righthand side $f=\epsilon$ was discussed. Here we give a new argument with more streamlined computations.  This also covers the Gursky-Streets equation when $k=1$.

When $k=n$, the operator
\[
F_n(r)=r_{00}\sigma_n(R)-(T_{n-1}(R), r_{0i}\otimes r_{0i})=\sigma_{n+1}(r),
\]
hence it is just the famous Monge-Ampere operator. It is not hard to see that the theory of Monge-Ampere equation can be used directly to solve the equation
\[
F_n(r)=f.
\]
We shall skip the details.

On the other hand, the Gursky-Streets equation becomes rather subtle when  $2\leq k\leq n-1$. When $k=2$, Gursky and Streets obtained a smooth solution with uniform $C^1$ bound for a perturbed equation \cite{GS}. Very recently, the second author solved the Gursky-Streets equation with uniform $C^{1, 1}$ bound, for $n\geq 4$. There are several subtle points. First of all, the concavity of the operator  $\log F_k (r)$ is rather subtle for $k=2$, and it is still unknown for $3\leq k\leq n-1$; see \cite{He17} for the discussion and the conjecture on the concavity. The estimate of second order, in particular $\Delta u$ appears to be very subtle. 

Lastly, we introduce a family of operators, which is the complex companion of $F_k$. Let $u: \R\times \C^n\rightarrow \R$ be a real valued function. Consider the following $(n+1)\times (n+1)$ matrix
 \[
r=\begin{pmatrix} r_{00} &r_{0i}\\
 \bar r_{0i} & R
\end{pmatrix}
\]
where $R$ is a $n\times n$ Hermitian matrix. We take $R=\p\bar\p u$ and 
\[
r=\begin{pmatrix} u_{tt} &\p u_t\\
 \bar \p u_t & \p\bar\p u
\end{pmatrix}
\]
Denote the operator, $1\leq k\leq n$,
\[
G_k(r)= u_{tt} \sigma_k(\p\bar\p u)-(T_{k-1}(\p\bar \p u), \p u_t\otimes \bar \p u_t),
\]
where $T_{k-1}(R)_{i\bar j}=\sigma_k(R) \frac{\p \log \sigma_k(R)}{\p R^{i\bar j}}$. 
When $k=1$, we get that
\[
G_1(r)=u_{tt}\Delta u-|\nabla u_t|^2
\]
is the Donaldson operator on $\R\times \C^n$. When $k=n$, 
\[
G_n(r)=u_{tt}\sigma_n (\p\bar\p u)-(T_{k-1}(\p\bar \p u), \p u_t\otimes \bar \p u_t)
\]
is a special case of the complex Monge-Ampere operator.  Actually this operator is the operator underline the geodesic equation in space of K\"ahler metrics,
\[
\phi_{tt}-|\nabla \phi|^2_{\omega_\phi}=0,
\]
which was studied extensively in literature. Similar as in \cite{He17}, we conjecture, 
\begin{conj}
For $2\leq k\leq n-1$, 
we conjecture that the operator
$\log G_k (r)$
is concave on $r$, for $r_{00}>0$, $\p\bar\p u$ in $\Gamma_k^+$ cone and $G_k(r)>0$. 
\end{conj}

\bibliographystyle{plain}

\end{document}